\documentclass[12pt,reqno]{amsart} 

\pdfoutput=1 
\usepackage[utf8]{inputenc}
\usepackage{multicol}

\usepackage[mathlines]{lineno}
\usepackage{microtype}

\usepackage{xcolor}

\definecolor{my-red}{rgb}{0.5,0.0,0.0}
\definecolor{my-blue}{rgb}{0.0,0.0,0.6}
\definecolor{my-green}{rgb}{0.0,0.5,0.0}
\definecolor{light-gray}{gray}{0.6}

\usepackage{scalerel}

\usepackage{enumitem}

\usepackage{tocvsec2}

\usepackage[labelfont=bf]{caption}

\usepackage[font={small},labelfont=md]{subfig}

\usepackage[noadjust,nosort]{cite}

\usepackage{amsmath, amsthm, amssymb, tikz, mathtools, array, mathrsfs, tensor, ifthen, xparse, graphicx}

\usepackage[foot]{amsaddr}

\usepackage{polynom}

\usepackage{esint}

\usepackage{multicol}

\usepackage{dsfont}
	\newcommand{\one}{\mathds{1}}

\usepackage{framed}

\usepackage{stackengine}

\usepackage{upref}

\usepackage[pagebackref=true, colorlinks=true, urlcolor=light-gray, linkcolor=my-red, citecolor=my-green]{hyperref}

\usepackage{orcidlink}
	
\numberwithin{equation}{section}
\numberwithin{equation}{section}


\newcommand{\eq}[1]{\begin{linenomath}\postdisplaypenalty=0\begin{align*} #1 \end{align*}\end{linenomath}} 

\newcommand{\eeq}[1]{\begin{linenomath}\postdisplaypenalty=0\begin{align} \begin{split} #1 \end{split} \end{align}\end{linenomath}}

\newcommand{\stackref}[2]{
\readlist*\mylist{#1}
\stackrel{\mbox{\footnotesize\foreachitem\x\in\mylist[]{\ifnum\xcnt=1\else,\fi\eqref{\x}}}}{#2}
} 
\newcommand{\stackrefp}[2]{
\readlist*\mylist{#1}
\stackrel{\hphantom{\mbox{\footnotesize\foreachitem\x\in\mylist[]{\ifnum\xcnt=1\else,\fi\eqref{\x}}}}}{#2}
}
\newcommand{\stackrefpp}[3]{
\readlist*\mylist{#1}
\readlist*\mylistt{#2}
\stackrel{\parbox{\widthof{\footnotesize\foreachitem\x\in\mylistt[]{\ifnum\xcnt=1\else,\fi\eqref{\x}}}}{\centering\footnotesize\foreachitem\x\in\mylist[]{{\ifnum\xcnt=1\else,\fi\eqref{\x}}}}}{#3}
} 

\newcommand{\vphi}{\varphi}

\newcommand{\cC}{\mathcal{C}}

\newcommand{\cG}{\mathcal{G}}

\newcommand{\cR}{\mathcal{R}}

\DeclareMathOperator{\id}{id}

            \makeatletter
            \DeclareFontFamily{OMX}{MnSymbolE}{}
            \DeclareSymbolFont{MnLargeSymbols}{OMX}{MnSymbolE}{m}{n}
            \SetSymbolFont{MnLargeSymbols}{bold}{OMX}{MnSymbolE}{b}{n}
            \DeclareFontShape{OMX}{MnSymbolE}{m}{n}{
                <-6>  MnSymbolE5
               <6-7>  MnSymbolE6
               <7-8>  MnSymbolE7
               <8-9>  MnSymbolE8
               <9-10> MnSymbolE9
              <10-12> MnSymbolE10
              <12->   MnSymbolE12
            }{}
            \DeclareFontShape{OMX}{MnSymbolE}{b}{n}{
                <-6>  MnSymbolE-Bold5
               <6-7>  MnSymbolE-Bold6
               <7-8>  MnSymbolE-Bold7
               <8-9>  MnSymbolE-Bold8
               <9-10> MnSymbolE-Bold9
              <10-12> MnSymbolE-Bold10
              <12->   MnSymbolE-Bold12
            }{}
            
            \let\llangle\@undefined
            \let\rrangle\@undefined
            \DeclareMathDelimiter{\llangle}{\mathopen}%
                                 {MnLargeSymbols}{'164}{MnLargeSymbols}{'164}
            \DeclareMathDelimiter{\rrangle}{\mathclose}%
                                 {MnLargeSymbols}{'171}{MnLargeSymbols}{'171}
            \makeatother
            
    \DeclareFontFamily{U}{matha}{\hyphenchar\font45}
    \DeclareFontShape{U}{matha}{m}{n}{ <-6> matha5 <6-7> matha6 <7-8>
    matha7 <8-9> matha8 <9-10> matha9 <10-12> matha10 <12-> matha12 }{}
    \DeclareSymbolFont{matha}{U}{matha}{m}{n}
    \DeclareFontFamily{U}{mathx}{\hyphenchar\font45}
    \DeclareFontShape{U}{mathx}{m}{n}{ <-6> mathx5 <6-7> mathx6 <7-8>
    mathx7 <8-9> mathx8 <9-10> mathx9 <10-12> mathx10 <12-> mathx12 }{}
    \DeclareSymbolFont{mathx}{U}{mathx}{m}{n}
    
    \DeclareMathDelimiter{\llbrack} {4}{matha}{"76}{mathx}{"30}
    \DeclareMathDelimiter{\rrbrack} {5}{matha}{"77}{mathx}{"38}

\newcommand{\rev}[1]{\rho( #1 )}
\newcommand{\semi}[1]{\overline{#1}}

\newcommand{\applies}{\leadsto}
\newcommand{\seqnum}[1]{\href{https://oeis.org/#1}{\rm \color{my-blue}\underline{#1}}}
\newcommand{\parng}[2]{#1\hspace{0.6pt}{\boldsymbol:}\hspace{0.3pt}#2}
\newcommand{\dif}{\scaleobj{0.9}{\Delta}}
\newcommand{\cat}{\oplus}

\usepackage{cancel}
\usepackage{tikz}

\newtheorem{theorem}{Theorem}[section]
\newtheorem{proposition}[theorem]{Proposition}

\newtheorem{lemma}[theorem]{Lemma}

\theoremstyle{definition} 
\newtheorem{definition}[theorem]{Definition}

\newtheorem{remark}[theorem]{Remark}

\title[Partial sums of $m$-step Fibonacci numbers]{A new combinatorial interpretation of partial sums of $m$-step Fibonacci numbers}



\author[E. Bates]{Erik Bates$^*$\,\orcidlink{0000-0002-3472-036X}}
\email{ebates@ncsu.edu}
\author[B. Morrison]{Blan Morrison$^*$}
\email{bhmorris@ncsu.edu}
\address{$^*$Department of Mathematics, North Carolina State University}

\author[M. Rogers]{Mason Rogers$^{**}$}
\address{$^{**}$Department of Earth, Atmospheric and Planetary Sciences, Massachusetts Institute of Technology}
\email{masonr@mit.edu}

\author[A. Serafini]{Arianna Serafini$^\dagger$}
\address{$^\dagger$Bridgewater Associates}
\email{aserafini@alumni.stanford.edu}

\author[A. Sood]{Anav Sood$^\ddagger$\,\orcidlink{0000-0003-0833-1237}}
\address{$^\ddagger$Department of Statistics, Stanford University}
\email{anavsood@stanford.edu}

\begin{document}

\begin{abstract}
The sequence of partial sums of Fibonacci numbers, beginning with $2$, $4$, $7$, $12$, $20$, $33,\dots$, has several combinatorial interpretations (OEIS \seqnum{A000071}).  
For instance, the $n$-th term in this sequence is the number of length-$n$ binary words that avoid $110$.  
This paper proves a related but new interpretation: given a length-$3$ binary word---called the keyword---we say two length-$n$ binary words are equivalent if one can be obtained from the other by some sequence of substitutions: each substitution replaces an instance of the keyword with its negation, or vice versa.  
We prove that the number of induced equivalence classes is again the $n$-th term in the aforementioned sequence.
When the keyword has length $m+1$ (instead of $3$), the same result holds with $m$-step Fibonacci numbers.
What makes this result surprising---and distinct from the previous interpretation---is that it does not depend on the keyword, despite the fact that the sizes of the equivalence classes do.
On this final point, we prove several results on the structure of equivalence classes, and also pose a variety of open problems.
\end{abstract}

\maketitle
\thispagestyle{empty}
\setcounter{tocdepth}{1}
\tableofcontents


\section{Introduction}
\subsection{Main result}
The primary goal of this paper is to give a new combinatorial interpretation of the partial sums of $m$-step Fibonacci numbers. 
For any positive integer $m$, the $m$-step Fibonacci numbers are the sequence $(F_n^{(m)})_{n\ge0}$ defined by the initial condition
\eeq{ \label{fib_a}
F^{(m)}_0 = 1,\quad
F^{(m)}_n = 2^{n-1} \quad \text{for $n\in\{1,\dots,m-1\}$,}
}
together with the recurrence
\eeq{ \label{fib_b}
F^{(m)}_{n} = F^{(m)}_{n-1} + \cdots + F^{(m)}_{n-m} \quad \text{for $n\ge m$}.
}
The usual Fibonacci sequence (OEIS~\seqnum{A000045} with index shifted by one) is the case $m=2$.
Typically the case $m=3$ is called tribonacci (OEIS~\seqnum{A000073} with index shifted by two), the case $m=4$ is called tetranacci (OEIS~\seqnum{A000078} with index shifted by three), and so on.
The trivial scenario $m=1$ has $F^{(1)}_n = 1$ for all $n\ge0$ (OEIS~\seqnum{A000012}).

We will realize the quantity $F_0^{(m)}+F_1^{(m)}+\cdots+F_n^{(m)}$ combinatorially, from an equivalence relation $\sim_a$ on $\{0,1\}^n$ defined as follows.
For a binary word $a\in\{0,1\}^{m+1}$, let $\neg a$ denote its (bitwise) negation: every $0$ becomes $1$, and every $1$ becomes $0$. 
Upon fixing the word $a$ and some positive integer $n$, we induce an equivalence relation on $\{0,1\}^n$ by declaring that an appearance of $a$ as a subword can be replaced with $\neg a$, and vice versa.
Extending this identification transitively (i.e.~allowing multiple replacements done one at a time), we denote the resulting equivalence relation by $\sim_{a}$.
We call $a$ the \textit{keyword}.

For example, if $a=101$, then $\neg a = 010$, and the following three words of length $n=5$ are equivalent:
\begin{subequations}
\label{size4class}
\eeq{ \label{size4class_a}
10\underset{\underset{\substack{\clap{\textsf{replace $\neg a$}} \\ \clap{\textsf{with $a$}}}}{\big\uparrow}}{\underline{010}}\ \ \sim_a\ \ 
\underset{\underset{\substack{\clap{\textsf{replace $a$}} \\ \clap{\textsf{with $\neg a$}}}}{\big\uparrow}}{\underline{101}}01\ \ \sim_a\ \ 01001.
}
In fact, there is a fourth equivalent word obtained by modifying the second replacement:
\eeq{ \label{size4class_b}
1\underset{\underset{\substack{\clap{\textsf{replace $\neg a$}} \\ \clap{\textsf{with $a$}}}}{\big\uparrow}}{\underline{010}}1\ \ \sim_a\ \ 11011.
}
\end{subequations}
These four words constitute one equivalence class in $\{0,1\}^5 / \sim_a$.
%
Our main result determines the total number of equivalence classes, and surprisingly it depends only on the length of $a$:

\begin{theorem}\label{wordversion}
For any $n,m\ge 1$ and any keyword $a\in\{0,1\}^{m+1}$, the number of equivalence classes on $\{0,1\}^n$ induced by the equivalence relation $\sim_a$ is equal to $F_0^{(m)}+F_1^{(m)}+\cdots+F_{n}^{(m)}$.
\end{theorem}

The proof of Theorem~\ref{wordversion} is given in Section~\ref{main_proof}.
The sizes and structure of equivalence classes is studied in Section~\ref{sec_size_structure}, and open problems are provided in Section~\ref{sec_open}.

\begin{remark}[Other relevant sequences for Theorem~\ref{wordversion}] \label{test_rmk}
For $m=2$, we have $F_0^{(2)}+F_1^{(2)}+\cdots+F_n^{(2)} = F_{n+2}^{(2)}-1$ (OEIS~\seqnum{A000071}).
So in this case, Theorem~\ref{wordversion} can be regarded as an interpretation of the Fibonacci sequence directly.
But this simplification does not exist for $m\ge3$; for example, see\ OEIS~\seqnum{A008937} for $m=3$.
For general $m$, the sequence $(F_0^{(m)}+F_1^{(m)}+\cdots+F_n^{(m)})_{n\ge0}$ is the $m$-th column of OEIS~\seqnum{A172119}.
\end{remark}

\subsection{Discussion of main result} \label{sec_discussion}
We first mention two easy cases of Theorem~\ref{wordversion}:
\begin{itemize}
\item If $n < m+1$, then the keyword $a\in\{0,1\}^{m+1}$ is too long to appear as a subword of any element of $\{0,1\}^n$, so the equivalence relation $\sim_a$ makes no nontrivial identifications.
That is, every equivalence class has exactly one element. 
This is consistent with the fact that $F_0^{(m)} + F_1^{(m)} + \cdots + F_n^{(m)} = 2^n$ for any $n\le m$.
\item If $n = m+1$, then the only element of $\{0,1\}^n$ containing $a$ as a subword is $a$ itself, so the only nontrivial identification is $a\sim_a \neg a$.
That is, there is one equivalence class with two elements, and all other equivalance classes are singletons.
This is consistent with the fact that $F_0^{(m)} + F_1^{(m)} + \cdots + F_{m+1}^{(m)} = 2^{m+1}-1$.
\end{itemize}

As soon as $n \ge m+2$, the sizes of equivalence classes vary with $a$.
For example, when $n=4$ and $m=2$, Theorem~\ref{wordversion} says there are $1 + 1 + 2 + 3 + 5 = 12$ equivalence classes, and Table~\ref{differentsizes} (top left) shows two different ways these equivalence classes can arrange themselves.
This variability makes Theorem~\ref{wordversion} all the more surprising, and invites questions on the sizes and structure of equivalence classes.
We offer a few answers in Section~\ref{sec_size_structure}, and list several open problems in Section~\ref{sec_open}.

\begin{table}[ht]
\caption{Number of equivalence classes of size $s$ induced on $\{0,1\}^n$ by two different choices of the keyword $a$.}
\setlength{\columnsep}{-2.5cm}
\begin{multicols}{2}
    \begin{tabular}{c||cc}
        $n=4$ & $a=110$ & $a=101$ \\\hline\hline
        $s=1$ & 8 & 10 \\
        $s=2$ & 4 & 0 \\
        $s=3$ & 0 & 2 \\\hline
        total & 12 & 12 \\
    \end{tabular}
 \\[0.3in]
    \begin{tabular}{c||cc}
        $n=5$ & $a=110$ & $a=101$ \\\hline\hline
        $s=1$ & 10 & 16 \\
        $s=2$ & 8 & 0 \\
        $s=3$ & 2 & 0 \\
        $s=4$ & 0 & 4 \\ \hline
        total & 20 & 20 \\
    \end{tabular}
    \\[0in]
\columnbreak
\vspace*{0.1in}
    \begin{tabular}{c||cc}
        $n=6$ & $a=110$ & $a=101$ \\\hline\hline
        $s=1$ & 12 & 26 \\
        $s=2$ & 12 & 0 \\
        $s=3$ & 8 & 0 \\
        $s=4$ & 1 & 0 \\
        $s=5$ & 0 & 6 \\
        $s=6$ & 0 & 0 \\
        $s=7$ & 0 & 0 \\
        $s=8$ & 0 & 1 \\ \hline
        total & 33 & 33 \\
    \end{tabular}
    \label{differentsizes}
\end{multicols}
\end{table}


When $a = 110$, the conclusion of Theorem~\ref{wordversion} can be inferred from Zeckendorf's theorem, which states that every nonnegative integer can be uniquely written as a sum of nonconsecutive Fibonacci numbers \cite{lekkerkerker52,brown64,zeckendorf72}.
To see the relevance of this fact, map the binary word $u = (u_1,\dots,u_n)\in\{0,1\}^n$ to the integer $N_u = \sum_{i=1}^n u_i F_i^{(2)}$.
This mapping is not injective, thanks to the Fibonacci recursion $F_{i}^{(2)}+F_{i+1}^{(2)}=F_{i+2}^{(2)}$ which is encoded by the relation $110\sim_a 001$.
In this perspective, the equivalence class of $u$ under $\sim_a$ consists of different representations of $N_u$ as a sum of distinct elements of $\{F_1^{(2)},\ldots,F_n^{(2)}\}$.
Therefore, the number of equivalence classes is at least the number of integers between $0 = N_{0\cdots0}$ and $F_1^{(2)}+\cdots+F_n^{(2)} = N_{1\cdots1}$.
For the reverse inequality, one needs to check that no two equivalence classes correspond to the same integer.
Indeed, every equivalence class has a representative avoiding the subword $110$, and Zeckendorff's theorem tells us---after a little extra work to account for Remark~\ref{technical_annoyance} stated below---that every integer has only one such representation.

\begin{remark}[Technical point concerning the Fibonacci numeration system] \label{technical_annoyance}
Unless $N$ is one less than a Fibonacci number, its Zeckendorf representation is one bit longer than its shortest representation.
In other words, the `extra work' mentioned above is to compare words that avoid $110$ but end in $11$.
Moreover, to ensure that the equivalence class in $\{0,1\}^n$ corresponding to $N$ includes \textit{all} representations of $N$, one should insist that $N$ can be written in $n-1$ bits, i.e.~$N\le F_1^{(2)}+\cdots+F_{n-1}^{(2)} = F_{n+1}^{(2)}-2$.
The case $N=F_{n+1}^{(2)}-1$ is also allowed (despite needing all $n$ bits) because its representation is unique: see \cite[Thm.~5(a)]{klarner66} or \cite[Thm.~2]{carlitz68}.
\end{remark}

Unfortunately, the argument discussed above for Theorem~\ref{wordversion} does not extend to other keywords such as $a = 101$.
After all, the relation $101\sim_a010$ has no obvious connection to the Fibonacci recursion nor to a nice numeration system.
Moreover, Zeckendorf's theorem does not seem to play a role in the structure of equivalence classes: in the previous discussion every equivalence class had a \textit{unique} representative avoiding the keyword, but this no longer holds when the keyword is $a = 101$.
For instance, the equivalence class from \eqref{size4class} has two representatives that avoid $101$.
This begins to explain the different equivalence class structures seen in Table~\ref{differentsizes}, and also demonstrates the challenge of proving Theorem~\ref{wordversion} for an arbitrary keyword.

Although our main objective is Theorem~\ref{wordversion}, the reader may wonder at this point when two keywords yield the same sizes of equivalence classes.
For instance, it is immediate that $\sim_a$ and $\sim_{\neg a}$ are the same equivalence relation, so $a$ can be replaced with $\neg a$ without any effect.
In Proposition~\ref{thm_iso}, we identify two other keyword modifications that preserve equivalence class structure: reversal and ``seminegation'' (see Definition~\ref{def_seminegation}).
Together these operations can be used to transform any $a\in\{0,1\}^3$ to either $110$ or $101$, so these two keywords account for all possible equivalence class structures induced by length-$3$ keywords.
Similarly, any length-$4$ keyword can be transformed to either $0000$, $0001$, or $0011$, and these three keywords all have different equivalence class structures.
But for longer keywords, the story is more complicated: see open problem \ref{open1}.





\subsection{Related literature} \label{related_lit}

For the ``Fibonacci keyword'' $a = 110$, we have mentioned that each equivalence class corresponds to a nonnegative integer $N$, and the elements of the equivalence class are different representations of $N$ as a sum of distinct Fibonacci numbers.
A natural quantity to investigate is the number of such representations, denoted by $R(N)$.
By Remark~\ref{technical_annoyance}, if $N\le F_{n+1}^{(2)}-1$, then $R(N)$ is equal to the size of the equivalence class in $\{0,1\}^n/\sim_a$ corresponding to $N$.

The study of the sequence $(R(N))_{N\ge0}$ (OEIS~\seqnum{A000119}) dates back to at least \cite{hoggatt_basin63}.
Its value is known at various special values of $N$, based on a variety of recursions \cite{klarner66,klarner68,carlitz68,bicknelljohnson_fielder99}.
For general $N$, it is possible to express $R(N)$ as a product of $2\times2$ matrices determined by the Zeckendorf represenation of $N$ \cite{berstel01}.
This formula was generalized to $m$-step Fibonacci numbers in \cite{kocabova_masakova_pelantova07}.
An alternative formula was found in \cite{edson_zamboni06} using binomial coefficients modulo $2$.


To align with the Fibonacci numeration system, it is natural to partition the sequence $(R(N))_{N\geq1}$ into blocks of the form $(R(N))_{F_n^{(2)}\leq N<F^{(2)}_{n+1}}$.
It was shown in \cite{edson_zamboni04} how to recursively compute each block from the previous one, which enabled the observation that the set $\{N:\, \text{$R(N)=s$ and $F_{n}^{(2)}\leq N<F_{n+1}^{(2)}$}\}$ has the same cardinality for all $n\ge2s$.
For example, this fact manifests in Table~\ref{differentsizes} as follows: the number of equivalence classes of size $s=1$ grows by $2$ each time, and the number of equivalence classes of size $s=2$ grows by $4$ each time.
For $m\ge 3$, these increments are no longer constant and not fully understood.
In the special case of $s=1$, it is known that the increments are $(m-1)$-step Fibonacci numbers \cite[Prop.~3.1]{kocabova_masakova_pelantova07}.

For other values of the keyword $a$, the induced equivalence classes do not seem to have been studied previously in the literature.
On the other hand, our pattern substitutions ($a\leftrightarrow\neg a$) can be thought of as a modification of pattern avoidance, which is a well-studied problem, e.g.\ \cite{noonan_zeilberger99,baccherini_merlini_sprugnoli07,merlini_sprugnoli11,bilotta_merlini_pergola12,bilotta_grazzini_pergola13,carrigan_hollars_rowland24}; see also \cite[Chap.~8]{sedgewick_flajolet13}.
In \cite{guibas_odlyzko81}, Guibas and Odlyzko computed the generating function for the number of length-$n$ words avoiding a given set of subwords.
For instance, avoiding both the subwords $a$ and $\neg a$ corresponds to size-$1$ equivalence classes in our setting; see Proposition~\ref{guibas_rem} for a precise consequence of this connection to \cite{guibas_odlyzko81}.


\section{Proof of main result} \label{main_proof}

\subsection{Notation and definitions} \label{def_not} 

An element $u=(u_1,\dots,u_n)$ of $\{0,1\}^n$ is called a \textit{(binary) word of length $n$}; the coordinate $u_i$ is said to be $i$-th \textit{letter} of $u$.
The subword of $u$ starting at letter $i$ and ending at letter $j$ is denoted by $u_{\parng{i}{j}} = (u_i,\dots,u_j)$.
An interval of integers is written as $\llbrack i,j\rrbrack = \{i,i+1,\dots,j\}$.
The concatenation of $u\in\{0,1\}^n$ and $w\in\{0,1\}^{n'}$ is denoted by
\eq{
u\cat w = (u_1, \dots, u_n, w_1, \dots, w_{n'})\in\{0, 1\}^{n+n'}.
}
In what follows, 
$a$ is a fixed element of $\{0,1\}^{m+1}$ referred to as the \textit{keyword}.

Our goal in this subsection is to formalize the equivalence relation appearing in Theorem~\ref{wordversion}.
The first step is to make the following definition.

\begin{definition}[Substitutions involving the keyword] \label{applicability}
Given a word $u\in\{0,1\}^n$ and a keyword $a\in\{0,1\}^{m+1}$,
we say that $a$ is \textit{applicable} to $u$ at letter $i$ (and write $a\applies_i u$) if
either $a$ or its negation appears as a subword of $u$ starting at letter $i$:
\eq{
    u_{\parng{i}{i+m}} \in \{a,\neg a\}.
}
(Notice that this requires $i+m\le n$.)
If $a$ is not applicable to $u$ at letter $i$, then we write $a\not\applies_i u$.

Let $\vphi_i^{(a)}\colon \{0,1\}^n\to \{0,1\}^n$ be the map that replaces $a$ with $\neg a$ or vice versa if one of them appears as a subword starting at letter $i$, and does nothing otherwise.
More precisely:
\begin{subequations} \label{def_simple_map}
\begin{align}
&\text{If $a\applies_i u$, then $\vphi_{i}^{(a)}(u)=u'$, where $u'_j=\begin{cases} 
      \neg u_{j} & j\in\llbrack i,i + m\rrbrack \\
      u_{j} & \text{otherwise.}
   \end{cases}$} \\
&\text{If $a\not\applies_i u$, then $\vphi_i^{(a)}(u)=u$.}
\end{align}
\end{subequations}
We call $\vphi_i^{(a)}$ a \textit{simple map of index $i$}. 
We refer to the set of indices $\llbrack i,i+m\rrbrack$ as the \textit{line of action} of $\vphi_i^{(a)}$.
\end{definition}

Simple maps do not commute in general, since one substitution could prevent another.
For example, if $a = 101$, then
\eq{
&(\vphi^{(a)}_2\circ\vphi^{(a)}_1)(\underline{010}10)
= \vphi^{(a)}_2(10110) = 10110, \\
\text{while} \quad
&(\vphi^{(a)}_1\circ\vphi^{(a)}_2)(0\underline{101}0)
= \vphi^{(a)}_1(00100) = 00100.
}
Of course, it can also be the case that a certain substitution is only possible after a substitution of a different index takes place.
For example (again with $a=101$):
\eeq{ \label{a7v2}
&(\vphi^{(a)}_3\circ\vphi^{(a)}_1)(\underline{101}10)
= \vphi^{(a)}_3(01\underline{010}),
= 01101 \\
\text{while} \quad
&(\vphi^{(a)}_1\circ\vphi^{(a)}_3)(10110)
= \vphi^{(a)}_1(\underline{101}10)
= 01010.
}
What is clear, however, is that simple maps with disjoint lines of action \textit{do} commute, since any substitution made by one does not interfere with a possible substitution by the other.
We record this fact for later use:
\eeq{ \label{simple_commute}
\llbrack i,i+m\rrbrack \cap \llbrack j,j+m\rrbrack = \varnothing
\quad \implies \quad
\vphi_i^{(a)} \circ \vphi_j^{(a)} 
= \vphi_j^{(a)} \circ \vphi_i^{(a)}.
}
Proposition~\ref{commute_lem} gives necessary and sufficient conditions for commutativity, but for now it is enough to know \eqref{simple_commute}.

Note that all simple maps are involutions; in particular, they are invertible.
This explains why $\sim_a$ defined below is an equivalence relation.

\begin{definition}[Equivalence relation induced by the keyword]
Given a keyword $a\in\{0,1\}^{m+1}$, we write $u\sim_a v$ 
if there exists a sequence of indices $(i_1,\dots,i_r)$ such that $(\vphi_{i_r}^{(a)} \circ \dots \circ \vphi_{i_1}^{(a)})(u) = v$.
We allow the empty composition (i.e.~the identity map when $r=0$) so that $u\sim_a u$.
\end{definition}

Transitivity of $\sim_a$ is seen by composing compositions, and symmetry follows from the fact that each simple map $\vphi_i^{(a)}$ is its own inverse.
Therefore, $\sim_a$ is indeed an equivalence relation on $\{0,1\}^n$.
We denote the equivalence class of $u\in\{0,1\}^n$ by
\eq{
[u]_a = \{v\in\{0,1\}^n:\, u\sim_a v\}.
}

In Section~\ref{sec_size_structure} we will consider the effect of changing the keyword.
But for now we keep the value of $a\in\{0,1\}^{m+1}$ fixed, so for the remainder of Section~\ref{main_proof} we dispense with notational decoration and just write $\vphi_i$ for $\vphi_i^{(a)}$, $\sim$ for $\sim_a$, and $[u]$ for $[u]_a$.

\subsection{Preliminary observation}




A key tool for proving Theorem~\ref{wordversion} is the following lemma about concatenations.

\begin{lemma}
\label{tack_on}
For any $u,v\in\{0,1\}^n$ and $w\in\{0,1\}^{n'}$, the following equivalences hold:
\begin{align}
u \sim v &\iff u\cat w \sim v\cat w, \label{tack_onb} \\
u \sim v &\iff w\cat u \sim w\cat v. \label{tack_onf}
\end{align}
\end{lemma}

The proof of Lemma~\ref{tack_on} requires one additional definition and one additional lemma.

\begin{definition}[Action by lists of simple maps]
A list of simple maps $\Phi = (\vphi_{i_r},\ldots,\vphi_{i_1})$ can be treated as a function $\{0,1\}^n\to\{0,1\}^n$ in the natural way:
$\Phi(u) = (\vphi_{i_r}\circ\cdots\circ\vphi_{i_1})(u)$.
We say $\Phi$ \textit{acts completely} on $u$ if each simple map performs an actual substitution:
\eeq{ \label{acts_completely}
(\vphi_{i_q} \circ\dots\circ \vphi_{i_1})(u)\neq (\vphi_{i_{q-1}} \circ\cdots\circ\vphi_{i_1})(u) \quad \text{for each $q\in\{1,\dots,r\}$}.
}
\end{definition}

\begin{lemma} \label{tgv9kd}
Fix $i\in\llbrack1,n\rrbrack$.
Consider a list of simple maps $\Upsilon = (\vphi_{j_r},\dots,\vphi_{j_1})$ that satisfies
all three of the following conditions:
\begin{enumerate}[label=\textup{(\roman*)}]

\item \label{tgv9kd_0} $j_1,\dots,j_r>i$

\item \label{tgv9kd_1} $\vphi_i \circ \Upsilon \neq \Upsilon\circ\vphi_i$ as maps on $\{0,1\}^n$.

\item \label{tgv9kd_2} There is some $x\in\{0,1\}^n$ such that $(\vphi_i,\Upsilon,\vphi_i)$ acts completely on $x$. 
\end{enumerate}
Then $\Upsilon$ contains an even number of instances of $\vphi_j$, where $j = \min\{j_1,\dots,j_r\}$.
\end{lemma}

\begin{proof}
We first claim that $\vphi_{i}$ and $\vphi_{j}$ have intersecting lines of action, i.e.~$\llbrack i,i+m\rrbrack\cap\llbrack j,j+m\rrbrack\neq\varnothing$.
Indeed, if this were not the case, then the line of action $\vphi_i$ would be disjoint from that of any simple map in $\Upsilon$, by \ref{tgv9kd_0} and minimality of $j$.
Hence $\vphi_{i}$ would commute with $\Upsilon$ by \eqref{simple_commute}, contradicting \ref{tgv9kd_1}.
So we assume henceforth that $\vphi_i$ and $\vphi_j$ have intersecting lines of action.

Since $\vphi_i(x)\neq x$ by \ref{tgv9kd_2}, we have
\eeq{ \label{bgqc7}
x_{\parng{i}{i+m}} \in\{a,\neg a\}.
}
The application of $\Upsilon\circ \vphi_i$ to $x$ negates the $i$-th letter only once (namely, in the application of $\vphi_i$), since $j_1,\dots,j_r > i$.
Since we have assumed $\vphi_{i}$ and $\vphi_{j}$ have intersecting lines of action, the $j$-th letter of $x$ is also negated by $\vphi_{i}$.
The $j$-th letter is further negated by $\vphi_{j_q}$ in $(\vphi_i\circ\vphi_{j_r}\circ\cdots\circ\vphi_{j_1}\circ\vphi_{i})(x)$
whenever $j_q = j$, but not when $j_q > j$.
So if the lemma were false (meaning there is an odd number of $q$'s such that $j_q = j$), then the $j$-th letter of $x$ would be negated an even number of times by $\Upsilon \circ \vphi_{i}$, while the $i$-th letter is negated an odd number of times (in fact, only once) by $\Upsilon \circ \vphi_{i}$. 
This would imply the scenario of \eqref{bgqc7} is no longer true after the application of $\Upsilon \circ \vphi_i$:
\eq{
((\Upsilon\circ \vphi_i)(x))_{\parng{i}{i+m}} \notin \{a,\neg a\}.
} 
This in turn means $\vphi_i$ has no effect on $(\Upsilon\circ \vphi_i)(x)$, which contradicts \ref{tgv9kd_2}.
\end{proof}

\begin{proof}[Proof of Lemma~\ref{tack_on}]
We will just show \eqref{tack_onf}, since it is notationally easier. 
To see that \eqref{tack_onb} follows, simply reverse the letters in all the words (including the keyword), apply \eqref{tack_onf}, and then reverse back.

The $\implies$ direction of \eqref{tack_onf} is trivial by just shifting indices and never modifying the prefix $w$ of length $n'$:
\eq{
\text{if}\quad (\vphi_{i_r}\circ\cdots\circ\vphi_{i_1})(u) = v, \quad \text{then} \quad
(\vphi_{i_r+n'}\circ\cdots\circ\vphi_{i_1+n'})(w\cat u) = w\cat v.
}
For the $\impliedby$ direction, we suppose $w\cat u \sim w\cat v$. 
That is, there exists a list of simple maps $\Phi = (\vphi_{i_r},\dots,\vphi_{i_1})$ such that 
\eeq{ \label{wuwv}
(\vphi_{i_r}\circ\dots\circ\vphi_{i_1})(w\cat u)=w\cat v.
}
We choose $r$ minimally 
in the sense that no shorter list satisfies \eqref{wuwv}. 
In particular, $\Phi$ acts completely on $w\cat u$:
\eeq{ \label{djb6c2}
w\cat u \neq \vphi_{i_1}(w\cat u) \neq (\vphi_{i_2}\circ\vphi_{i_1})(w\cat u)\neq\cdots\neq(\vphi_{i_r}\circ\cdots\circ\vphi_{i_1})(w\cat u).
}
If $i_q > n'$ for all $q\in\llbrack1,r\rrbrack$, then none of letters in the prefix $w\in\{0,1\}^{n'}$ are modified in 
\eqref{wuwv}.
In this case, a shift of indices results in
\eq{
(\vphi_{i_r - n'} \circ \dots \circ \vphi_{i_1 - n'})(u) = v,
}
so $u\sim v$ as desired.
In fact, the remainder of the proof will show that this is the only way \eqref{wuwv} can occur (when $r$ is minimal).
So suppose toward a contradiction that $i_q \le n'$ for some $q$. 

Let $i_{(1)}=\min\{i_1,\dots, i_r\}\le n'$. 
By minimality of $i_{(1)}$, the application of $\vphi_{i_q}$ in \eqref{wuwv} negates the $i_{(1)}$-th letter if and only if $i_q = i_{(1)}$.
In total, the $i_{(1)}$-th letter is negated an even number of times since $(w\cat u)_{i_{(1)}} = w_{i_{(1)}} = (w\cat v)_{i_{(1)}}$.
So there must be an even number of instances of $\vphi_{i_{(1)}}$ in \eqref{wuwv}:
\eq{
\Phi = \big(\dots, \underbrace{ \hspace{2.5pt} \vphi_{i_{(1)}}, \dots, \vphi_{i_{(1)}}, \dots, \vphi_{i_{(1)}}, \dots \vphi_{i_{(1)}} }_{ \text{even number of} \hspace{3pt} \vphi_{i_{(1)}}},\Psi_1\big).
}
Here $\Psi_1$ is a (possibly empty) list not containing any instance of $\vphi_{i_{(1)}}$. Consider the rightmost pair of $\vphi_{i_{(1)}}$ and denote the sequence between them as $\Phi_1$ like so: 
\eq{
\Phi = \big(\dots, \vphi_{i_{(1)}}, \Phi_1, \vphi_{i_{(1)}}, \Psi_1\big).
}

Now we check that the three hypotheses of Lemma~\ref{tgv9kd} are satisfied with $i = i_{(1)}$ and $\Upsilon = \Phi_1$.
By construction, each simple map in $\Phi_1$ has index strictly greater than $i_{(1)}$, so assumption~\ref{tgv9kd_0} is valid.
Second, $\Phi_1$ does not commute with $\vphi_{i_{(1)}}$ since $\vphi_{i_{(1)}}\circ \vphi_{i_{(1)}}=\id$ and $r$ is minimal for \eqref{wuwv}, so assumption~\ref{tgv9kd_1} is valid
(in particular, $\Phi_1$ is not empty).
Finally, assumption~\ref{tgv9kd_2} is valid with $x = \Psi_1(w\cat u)$ because of \eqref{djb6c2}.
The conclusion of Lemma~\ref{tgv9kd} is that there are at least two instances of $\vphi_{i_{(2)}}$ in $\Phi_1$, where $i_{(2)} > i_{(1)}$ is the smallest index of the simple maps in $\Phi_1$.

We now march inductively toward the desired contradiction.
Isolate the rightmost pair of $\vphi_{i_{(2)}}$ within $\Phi_1$:
\eq{
\Phi_1 = \big(\dots,\vphi_{i_{(2)}},\Phi_2,\vphi_{i_{(2)}}, \Psi_2\big),
}
where $\Psi_2$ and $\Phi_2$ contain no instances of $\vphi_{i_{(2)}}$.
The argument of the previous paragraph goes through verbatim with $(i_{(2)},\Phi_2)$ replacing $(i_{(1)},\Phi_1)$, and $x = (\Psi_2\circ\vphi_{i_{(1)}}\circ\Psi_1)(w\cat u)$.
Therefore, there are at least two instances of $\vphi_{i_{(3)}}$ in $\Phi_2$, where $i_{(3)} > i_{(2)}$ is the smallest index of the simple maps in $\Phi_2$, and so on.
Repeating this procedure, we ultimately find a list $\Phi_{r}$ that is empty (since $i_{(1)} < \cdots < i_{(r)}$ have exhausted all available indices) but does not commute with $\vphi_{i_{(r)}}$ (again by minimality of $r$).
This is an obvious contradiction and completes the proof.
\end{proof}

\subsection{Counting equivalence classes}

Here is an elementary lemma to set up our proof of the main result.

\begin{lemma} \label{nd6g3}
Let $(S_n)_{n\geq0}$ be a sequence of real numbers.
The following two statements are equivalent:
\begin{enumerate}[label=\textup{(\alph*)}]

\item \label{nd6g3_a} $S_n = F_0^{(m)}+F_1^{(m)}+\cdots+F_n^{(m)}$ for every $n\ge0$.

\item \label{nd6g3_b} $S_n = 2^n$ for all $n\in\llbrack0,m\rrbrack$, and $S_n = 2S_{n-1} - S_{n-m-1}$ for all $n\ge m+1$.

\end{enumerate}
\end{lemma}

\begin{proof}
\ref{nd6g3_a}$\implies$\ref{nd6g3_b}:
Assume $S_n = F_0^{(m)}+F_1^{(m)}+\cdots+F_n^{(m)}$.
In particular, for $n\in\llbrack0,m\rrbrack$ we have
\eq{
S_n \stackref{fib_a}{=} 1 + 2^0 + \cdots + 2^{n-1} = 2^n.
}
For $n\ge m+1$, we recover the desired recursion as follows:
\eq{
S_n = S_{n-1} + F_n^{(m)}
&\stackref{fib_b}{=} S_{n-1} + F_{n-1}^{(m)} + \cdots + F_{n-m}^{(m)} \\
&\stackrefp{fib_b}{=} S_{n-1} + S_{n-1} - S_{n-m-1}. 
}
\ref{nd6g3_b}$\implies$\ref{nd6g3_a}: 
We just showed that that the sequence $(F_0^{(m)}+F_1^{(m)}+\cdots+F_n^{(m)})_{n\ge0}$ satisfies statement~\ref{nd6g3_b}.
On the other hand, statement~\ref{nd6g3_b} uniquely determines the sequence $(S_n)_{n\ge0}$.
Therefore, we must have $S_n =  F_0^{(m)}+F_1^{(m)}+\cdots+F_n^{(m)}$.
\end{proof}

\begin{proof}[Proof of Theorem~\ref{wordversion}]
Denote the set of equivalence classes in $\{0,1\}^n$ by $\cC_n$.
We wish to show $|\cC_n| = F_0^{(m)}+\cdots+F_n^{(m)}$.
For $n<m+1$, the keyword $a\in\{0,1\}^{m+1}$ is too long to appear in any word, so $|\cC_n| = 2^n$.
By Lemma~\ref{nd6g3}, it now suffices to show
\eeq{\label{recursion}
|\cC_{n}|=2|\cC_{n-1}| - |\cC_{n - m-1}| \quad \text{for all $n\ge m+1$},
}
where $\cC_0$ has a single element: the equivalence class of the empty word.
To this end, consider the equivalence classes in $\{0,1\}^{n-1}$:  
\eeq{ \label{ifb7x}
\cC_{n-1}=\big\{[u^{(1)}],\dots, [u^{(|\cC_{n-1}|)}]\big\}. 
}
Here each $u^{(k)}\in\{0,1\}^{n-1}$ is a representative (chosen arbitrarily) of a different equivalence class. 
We claim that all the equivalence classes of length-$n$ words can be found in the set 
\eeq{ \label{vk7c4f}
\big\{[u^{(1)}\cat 0], [u^{(1)}\cat 1] ,\dots, [u^{(|\cC_{n-1}|)}\cat 0], [u^{(|\cC_{n-1}|)}\cat 1]\big\}.
}
To see this, consider any $v\in\{0,1\}^n$. 
We know from \eqref{ifb7x} that $v_{\parng{1}{n-1}}\sim u^{(k)}$ for some $k\in\llbrack1,|\cC_{n-1}|\rrbrack$. 
By the $\implies$ direction of \eqref{tack_onb}, if the last letter of $v$ is $v_n=0$, then $v\sim u^{(k)}\cat 0$.
If instead $v_n = 1$, then $v\sim u^{(k)}\cat 1$.
We have thus shown that \eqref{vk7c4f} is equal to $\cC_n$.

We wish to determine the cardinality of $\cC_n$, but the list in \eqref{vk7c4f} contains duplicates.
To count the number of the duplicates, we make two observations:
\begin{itemize}
    \item If two equivalence classes listed in \eqref{vk7c4f} are equal, then they are of the form  $[u^{(k)}\cat 0], [u^{(\ell)}\cat 1]$ for some $k \neq \ell$.
    This is because we have assumed $u^{(k)}\not\sim u^{(\ell)}$ whenever $k\neq \ell$, so $u^{(k)}\cat 0\not\sim u^{(\ell)}\cat 0$ and $u^{(k)}\cat 1\not\sim u^{(\ell)}\cat 1$ by the $\impliedby$ direction of \eqref{tack_onb}.
    \item No three elements in the list are equal to each other. 
    To see this, suppose $[u^{(k)}\cat 0]=[u^{(\ell)}\cat 1]$. 
    Then for any $h\in\llbrack1,|\cC_{n-1}|\rrbrack\setminus\{k,\ell\}$, we know from the previous bullet that $[u^{(h)}\cat 0]\neq[u^{(k)}\cat 0]$ and $[u^{(h)}\cat 1]\neq[u^{(\ell)}\cat 1]$.
\end{itemize}
Now define the set of ordered pairs yielding a repeat:
\eq{
\cR = \big\{(k,\ell): [u^{(k)}\cat 0] = [u^{(\ell)}\cat 1]\big\}.
}
It follows from the two observations above that $|\cC_{n}| = 2|\cC_{n-1}| - |\cR|$. 
To complete the proof of \eqref{recursion}, we will exhibit a bijection $f\colon\cR\to\cC_{n-m-1}$.

To define the bijection, consider some $(k,\ell) \in \cR$. 
Then $u^{(k)}\cat 0 \sim u^{(\ell)}\cat 1$, so there exists a list of simple maps $\Phi = (\vphi_{i_r},\dots,\vphi_{i_1})$ such that 
\eeq{ \label{ex3v}
(\vphi_{i_r}\circ\cdots\circ\vphi_{i_1})(u^{(k)}\cat 0)= u^{(\ell)}\cat 1.
}
In particular, since $u^{(k)}\cat 0$ ends in $0$ while $u^{(\ell)}\cat 1$ ends in $1$, there must be some instance of $\vphi_{n-m}$ in $\Phi$ in order to act on the $n$-th letter. 
By looking either immediately before or immediately after applying $\vphi_{n-m}$ for the first time in \eqref{ex3v}, we find some word $x\in\{0,1\}^n$ that is equivalent to $u^{(k)}\cat 0 \sim u^{(\ell)}\cat 1$ and ends in the keyword $a$:
\begin{subequations}
\label{wdef}
\begin{align}
[x] &= [u^{(k)}\cat 0] = [u^{(\ell)}\cat 1], \label{wdef_a} \\
\text{and} \quad x &= (x_1, x_2, \dots, x_{n-m-1}, a_1, \dots, a_{m+1}). \label{wdef_b}
\end{align}
\end{subequations}
Define $f\colon \cR\to \cC_{n-m-1}$ by $f(k,\ell)=[x_{\parng{1}{n-m-1}}]$.
(By the $\impliedby$ direction of \eqref{tack_onb}, any $x$ satisfying \eqref{wdef} will yield the same value of $[x_{\parng{1}{n-m-1}}]$, but we do not need this fact.) 

To show injectivity, suppose $f(k,\ell)=f(k',\ell')$, where $f(k',\ell')=[x'_{\parng{1}{n-m-1}}]$ for some $x'\in\{0,1\}^{n}$ that ends in $a$ and is equivalent to both $u^{(k')}\cat 0$ and $u^{(\ell')}\cat 1$.
We thus have
\eq{
[u^{(k)}\cat0] \stackref{wdef_a}{=}[x] 
&\stackref{wdef_b}{=} [x_{\parng{1}{n-m-1}}\cat a] \\
&\stackrefp{wdef_b}{=} [x'_{\parng{1}{n-m-1}}\cat a] \quad \text{by the $\implies$ direction of \eqref{tack_onb}} \\
&\stackrefp{wdef_b}{=} [x'] = [u^{(k')}\cat0].
}
Now the $\impliedby$ direction of \eqref{tack_onb} forces $u^{(k)} \sim u^{(k')}$.
But each $u^{(k)}$ is a representative of a distinct equivalence class in $\cC_{n-1}$, so we must have $k=k'$.
An analogous argument shows $\ell=\ell'$, so $f$ is injective.

To show surjectivity, consider an arbitrary equivalence class $[y]$ in $\cC_{n-m-1}$. 
Without loss of generality, assume the last letter of $a$ is $a_{m+1}=0$. 
(The case $a_{m+1}=1$ is handled analogously.)
Since the concatenation $y\cat a_{\parng{1}{m}}$ has length $n-1$, we know $y\cat a_{\parng{1}{m}}\sim u^{(k)}$ for some $k$. 
This implies $y\cat a \sim u^{(k)}\cat 0$ by the $\implies$ direction of \eqref{tack_onb}. 
By analogous reasoning, we also have $y\cat (\neg a) \sim u^{(\ell)}\cat 1$ for some $\ell$. 
Since $y\cat a \sim y\cat (\neg a)$, we conclude $[u^{(k)}\cat 0] = [u^{(\ell)}\cat 1]$. 
This means $(k,\ell)\in\cR$, so there exists $x\in\{0,1\}^n$ satisfying \eqref{wdef} such that $f(k,\ell) = [x_{\parng{1}{n-m-1}}]$.
Since $[x_{\parng{1}{n-m-1}}\cat a] = [x] = [u^{(k)}\cat 0] = [y\cat a]$, the $\impliedby$ direction of \eqref{tack_onb} gives $[x_{\parng{1}{n-m-1}}]=[y]$.
Hence $f(k,\ell) = [y]$, thereby proving surjectivity.
\end{proof}



\section{Sizes and structures of equivalence classes}\label{sec_size_structure}

As discussed in Section~\ref{sec_discussion}, the sizes of equivalence classes depend on the keyword $a\in\{0,1\}^{m+1}$, despite the total number of equivalence classes depending only on $m$.
Here we study these equivalence classes more systematically.
Section~\ref{sec_size1} examines size-$1$ equivalence classes, while Section~\ref{sec_structure} considers the graph-theoretic structure of larger equivalence classes.

\subsection{{Equivalence classes of size $1$}} \label{sec_size1}
The following result determines exactly when two keywords $a$ and $b$ yield the same number of size-$1$ equivalence classes.
To help parse statement~\ref{guibas_rem_b}, we point out that the condition $a_{\parng{1}{i}}\in\{a_{\parng{m+2-i}{m+1}}, \neg a_{\parng{m+2-i}{m+1}}\}$
simply means that the first $i$ letters of $a$ are the same as the last $i$ letters, or the same as the negation of the last $i$ letters.

\begin{proposition}\label{guibas_rem}
Let $a,b\in\{0,1\}^{m+1}$.
The following two statements are equivalent.
\begin{enumerate}[label=\textup{(\alph*)}]

\item \label{guibas_rem_a} For every $n$, the equivalence relations $\sim_a$ and $\sim_b$ on $\{0,1\}^n$ induce the same number of size-$1$ equivalence classes.

\item \label{guibas_rem_b} For every $i\in\llbrack1,m+1\rrbrack$, the following equivalence is true:
    \eeq{ \label{erfb8}
    a_{\parng{1}{i}}\in\{a_{\parng{m+2-i}{m+1}}, \neg a_{\parng{m+2-i}{m+1}}\} 
    \iff 
    b_{\parng{1}{i}}\in\{b_{\parng{m+2-i}{m+1}}, \neg b_{\parng{m+2-i}{m+1}}\} .
    }
\end{enumerate}
\end{proposition}

\begin{remark}[Other sizes of equivalence classes can still differ] \label{rem_uq1}
An example of $a,b$ satisfying statement \ref{guibas_rem_b} is $a = 10001$ and $b=01001$.
For both of these keywords, \eqref{erfb8} is true if and only if $i \in \{1,2,5\}$.
Hence Proposition~\ref{guibas_rem} guarantees that $\sim_a$ and $\sim_b$ induce the same number of equivalence classes of size $s=1$.
Nevertheless, Table~\ref{only1_table} shows that for some larger values of $s$, there are different numbers of size-$s$ equivalence classes.
These differences do not appear, however, until $n=12$.
\end{remark}

\begin{table}[ht]
\caption{Number of equivalence classes of size $s$ induced on $\{0,1\}^n$ by two different keywords ($a$ and $b$). 
By Proposition~\ref{guibas_rem}, these two keywords induce the same number of  equivalence classes of size $s=1$ (\textcolor{blue}{\textbf{bold blue}}), for every $n$.
But for larger $s$, the number of equivalence classes of size $s$ can still differ (\textcolor{red}{\textit{italic red}}).}
\label{only1_table}
\setlength{\columnsep}{-2cm}
\begin{multicols}{2}
    \begin{tabular}{c||cc}
        $n=11$ & $a=10001$ & $b=01001$ \\\hline\hline
        $s=1$ & \textcolor{blue}{\textbf{1,262}} & \textcolor{blue}{\textbf{1,262}}  \\
        $s=2$ & 256 & 256 \\
        $s=3$ & 80 & 80 \\
        $s=4$ & 6 & 6 \\
        $s=5$ & 2 & 2 \\\hline
        total & 1,606 & 1,606
    \end{tabular}
\columnbreak
\vspace*{-0.15in}
        \begin{tabular}{c||cc}
        $n=12$ & $a=10001$ & $b=01001$ \\\hline\hline
        $s=1$ & \textcolor{blue}{\textbf{2,356}} & \textcolor{blue}{\textbf{2,356}} \\
        $s=2$ & 528 & 528 \\
        $s=3$ & 176 & 176 \\
        $s=4$ & \textcolor{red}{\textit{26}} & \textcolor{red}{\textit{24}} \\
        $s=5$ & \textcolor{red}{\textit{8}} & \textcolor{red}{\textit{12}} \\
        $s=6$ & \textcolor{red}{\textit{2}} & \textcolor{red}{\textit{0}} \\\hline
        total & 3,096 & 3,096
    \end{tabular}
\end{multicols}
\end{table}

\begin{proof}[Proof of Proposition~\ref{guibas_rem}]
For a given keyword $a$, denote the number of size-$1$ equivalence classes by
\eq{
c^{(a)}_n = |\{u\in\{0,1\}^n:\, \text{$u$ avoids both the subwords $a$ and $\neg a$}\}|,
}
with the convention $c^{(a)}_0 = 1$.
Denote the generating function of $(c^{(a)}_n)_{n\ge0}$ by
\eeq{ \label{ib7q}
G^{(a)}(z) = \sum_{n=0}^\infty c^{(a)}_nz^{-n},
}
Following the strategy of \cite{guibas_odlyzko81}, we define the ``correlation polynomials''
\begin{subequations} \label{mq8c}
\begin{align}
P_1^{(a)}(z) &= \sum_{i=1}^m z^{m-i}\one\{a_{\parng{1}{i}}=a_{\parng{m+2-i}{m+1}}\},\\
P_2^{(a)}(z) &= \sum_{i=1}^m z^{m-i}\one\{a_{\parng{1}{i}}=\neg a_{\parng{m+2-i}{m+1}}\}.
\end{align}
\end{subequations}
In the notation of \cite{guibas_odlyzko81}, our $P_1^{(a)}(z)$ and $P_2^{(a)}(z)$ are their $AA_z = BB_z$ and $AB_z = BA_z$ respectively, where $A = a$ and $B = \neg a$. 
Furthermore, our $c_n^{(a)}$ and $G^{(a)}(z)$ are their $f(n)$ and $F(z)$, respectively.
Solving the linear system from \cite[Thm.~1]{guibas_odlyzko81} (with $q=2$) gives
\eeq{ \label{mb11z}
G^{(a)}(z) = \frac{z\big(P_1^{(a)}(z) + P_2^{(a)}(z)\big)}{(z-2)\big(P_1^{(a)}(z) + P_2^{(a)}(z)\big)+2}.
}
We thus have
\begin{align*}
c^{(a)}_n = c^{(b)}_n \quad \text{for all $n\ge0$}
\quad&\stackref{ib7q}{\iff}\quad G^{(a)} = G^{(b)} \\
&\stackref{mb11z}{\iff}\quad P_1^{(a)}+P_2^{(a)} = P_1^{(b)}+P_2^{(b)} \\
&\stackref{mq8c}{\iff}\quad \text{\eqref{erfb8} holds for all $i\in\llbrack1,m+1\rrbrack$}. \qedhere
\end{align*}
\end{proof}

Interestingly, the prefix-suffix comparison appearing in \eqref{erfb8} is also related to the commutativity of simple maps.
The following result tells us precisely when two simple maps commute on all of $\{0,1\}^n$.

\begin{proposition}
\label{commute_lem}
Given the keyword $a\in\{0,1\}^{m+1}$ and indices $i < j \le n-m$, let $\dif = j-i$.
There is commutativity of simple maps
\eeq{ \label{commute}
\vphi_i^{(a)} \circ \vphi_j^{(a)}=\vphi_j^{(a)} \circ \vphi_i^{(a)}
}
if and only if one of the following two conditions holds:
\begin{subequations}
\label{commute_reasons}
\begin{align}
\label{indices_commute} 
a_{\parng{1}{m+1-\dif}}\notin\{a_{\parng{\dif+1}{m+1}},\neg a_{\parng{\dif+1}{m+1}}\} \\
\text{or} \quad 
\label{disjoint_commute}
\text{$\vphi_i^{(a)}$ and $\vphi_j^{(a)}$ have disjoint lines of action, i.e.\ $\dif\ge m+1$}.
\end{align}
\end{subequations}
\end{proposition}

In words, \eqref{indices_commute} says the subword of length $m+1-\dif$ at the start of $a$ is \textit{not} equal to the subword of the same length at the end of $a$, nor equal to the latter subword's negation.
This is impossible if $\dif=m$, as the subword in question is a single letter $a_1$, which must be equal to either $a_{m+1}$ or $\neg a_{m+1}$.
Therefore, $\vphi_i^{(a)}$ and $\vphi_{i+m}^{(a)}$ are never fully commutative; \eqref{a7v2} is an example with $m=2$.

\begin{proof}[Proof of Proposition~\ref{commute_lem}]
First we prove \eqref{commute_reasons}$\implies$\eqref{commute}.
The implication \eqref{disjoint_commute}$\implies$\eqref{commute} was already noted in \eqref{simple_commute}.
So let us assume \eqref{indices_commute} and that the lines of action of $\vphi_i$ and $\vphi_j$ intersect, i.e.~$\dif\le m$.
We will verify the equality 
\eeq{ \label{7bjc4c}
(\vphi_i^{(a)}\circ \vphi_j^{(a)})(u) = (\vphi_j^{(a)}\circ \vphi_i^{(a)})(u)
}
for all $u\in\{0,1\}^n$ by considering three cases. 

\smallskip

\noindent \textbf{Case 1:} $a\not\applies_i u$ and $a\not\applies_j u$.

Both sides of \eqref{7bjc4c} are equal to $u$ in this case.

\smallskip

\noindent \textbf{Case 2:} $a\not\applies_i u$ and $a\applies_j u$.

Now the right-hand side of \eqref{7bjc4c} is equal to $\vphi_j^{(a)}(u)$.
For the left-hand side to agree, we need to show
$a\not\applies_i \vphi_j^{(a)}(u)$. 
Since $a\applies_j u$, we have $u_{\parng{j}{j+m}}\in\{a,\neg a\}$.
Omitting the last $\dif$ letters from this subword, we have
\begin{align*}
    u_{\parng{j}{i+m}}&\in\{a_{\parng{1}{m+1-\dif}},\neg a_{\parng{1}{m+1-\dif}}\}, \\
    \text{and thus} \quad (\vphi_j^{(a)}(u))_{\parng{j}{i+m}} = \neg u_{\parng{j}{i+m}} &\in \{\neg a_{\parng{1}{m+1-\dif}},a_{\parng{1}{m+1-\dif}}\}.
\intertext{Since we have assumed \eqref{indices_commute}, it follows that}
    (\vphi_j^{(a)}(u))_{\parng{j}{i+m}}
    &\notin\{\neg a_{\parng{\dif+1}{m+1}},a_{\parng{\dif+1}{m+1}}\},\\
    \text{and thus} \quad (\vphi_j^{(a)}(u))_{\parng{i}{i+m}}&\notin \{\neg a, a\}.
\end{align*}
So $a\not\applies_i \vphi_j^{(a)}(u)$, as desired.
\smallskip

\noindent \textbf{Case 3:} $a\applies_i u$.

In this case we have $u_{\parng{i}{i+m}}\in\{a,\neg a\}$. 
Omitting the first $\dif$ letters from this subword, we have
\begin{align}
\notag
u_{\parng{j}{i+m}}&\in\{a_{\parng{\dif+1}{m+1}},\neg a_{\parng{\dif+1}{m+1}}\}.
\intertext{Since we have assumed \eqref{indices_commute}, it follows that}
u_{\parng{j}{i+m}}&\notin\{\neg a_{\parng{1}{m+1-\dif}}, a_{\parng{1}{m+1-\dif}}\} \notag \\
\text{and thus}\quad u_{\parng{j}{j+m}}&\notin\{a,\neg a\}.
\label{nophi}
\intertext{In addition, $(\vphi_i^{(a)}(u))_{\parng{i}{i+m}} = \neg u_{\parng{i}{i+m}}\in\{\neg a,a\}$, so the exact same logic shows}
(\vphi_i^{(a)}(u))_{\parng{j}{j+m}}&\notin\{a,\neg a\}.
\label{yesphi}
\end{align}
Now we compare the two sides of \eqref{7bjc4c}.
On one hand, \eqref{nophi} says $a\not\applies_j u$, so the left-hand side of \eqref{7bjc4c} is  equal to $\vphi_i(u)$.
On the other hand, \eqref{yesphi} says $a\not\applies_j\vphi_i^{(a)}(u)$, so the right-hand side of \eqref{7bjc4c} is also equal to $\vphi_i^{(a)}(u)$. 
This completes the proof of \eqref{commute_reasons}$\implies$\eqref{commute}.

Now we prove \eqref{commute}$\implies$\eqref{commute_reasons} by contrapositive.
Suppose that $\vphi_i^{(a)}$ and $\vphi_j^{(a)}$ have intersecting lines of action (i.e.~$\dif\le m$) and that 
\eeq{\label{nocommute}
a_{\parng{1}{m+1-\dif}}\in\{a_{\parng{\dif+1}{m+1}},\neg a_{\parng{\dif+1}{m+1}}\}.
}
We will exhibit a word $u\in\{0,1\}^n$ such that $(\vphi_i^{(a)}\circ\vphi_j^{(a)})(u)\neq(\vphi_j^{(a)}\circ\vphi_i^{(a)})(u)$. 
Define the word $b\in\{0,1\}^{\dif}$ to be the last $\dif$ letters of either $a$ or $\neg a$, depending on the sign in \eqref{nocommute}:
\eq{
b = \begin{cases}
    a_{\parng{m+1-\dif+1}{m+1}} &\text{if $a_{\parng{1}{m+1-\dif}} = \neg a_{\parng{\dif+1}{m+1}}$} \\
    \neg a_{\parng{m+1-\dif+1}{m+1}} &\text{if $a_{\parng{1}{m+1-\dif}} = a_{\parng{\dif+1}{m+1}}$.}
\end{cases}
}
The definition of $b$ is meant to swap the cases in \eqref{nocommute} in the sense that
\begin{align}
\label{nv62}
a_{\parng{\dif+1}{m+1}}\cat b &\notin \{a,\neg a\}, \\
\label{nv63}
\text{whereas}\quad (\neg a_{\parng{\dif+1}{m+1}})\cat b &\in \{a,\neg a\}.
\end{align}
Now let $u\in\{0,1\}^n$ be any word such that 
\eq{
u_{\parng{i}{j+m}}=a \cat b.
}
Then clearly $a\applies_i u$, while $a\not\applies_j u$ since 
\eq{
u_{\parng{j}{j+m}} = a_{\parng{\dif+1}{m+1}} \cat b \stackref{nv62}{\notin}\{a,\neg a\}.
}
Thus $(\vphi_i^{(a)}\circ \vphi_j^{(a)})(u)=\vphi_i^{(a)}(u)$. 
On the other hand, we know $a\applies_j\vphi_i^{(a)}(u)$ since 
\eq{
(\vphi_i^{(a)}(u))_{\parng{j}{j+m}}
=\neg a_{\parng{\dif+1}{m+1}}\cat b
\stackref{nv63}{\in}\{a,\neg a\}.
}
Thus $(\vphi_j^{(a)}\circ \vphi_i^{(a)})(u)\neq \vphi_i^{(a)}(u)=(\vphi_i^{(a)}\circ \vphi_j^{(a)})(u)$, so we do not have commutativity.
\end{proof}

\subsection{Graph structure within equivalence classes} \label{sec_structure}

Each keyword induces a graph as follows.

\begin{definition} \label{def_graph}
Given the keyword $a\in\{0,1\}^{m+1}$, let $\cG_n^{(a)}$ denote the graph whose vertex set is $\{0,1\}^n$, and edge $\{u,v\}$ is present if and only if $v$ is obtained from $u$ by a \textit{single} substitution $a\mapsto \neg a$ or $\neg a\mapsto a$, i.e.~$\vphi_i^{(a)}(u) = v$ for some $i\in\llbrack1,n-m\rrbrack$.
\end{definition}

The connected components of $\cG_n^{(a)}$ are the equivalence classes under $\sim_a$, and the graph distance is
\eq{
d^{(a)}(u,v)\coloneqq\inf\big\{r:\, \exists\ \text{$i_1,\dots,i_r$ such that $(\vphi_{i_r}^{(a)} \circ \dots \circ \vphi_{i_1}^{(a)})(u)=v$}\big\}.
}
By convention we take $d^{(a)}(u,u)=0$.
Note that $d^{(a)}(u,v)<\infty$ if and only if $u\sim_a v$.
See Figure~\ref{graphex} for an example.

\begin{figure}[ht]
    \centering
\begin{tikzpicture} 
    \draw[fill=black] (-4,1) circle (3pt);
    \draw[fill=black] (-4,-1) circle (3pt);
    \draw[fill=black] (-2,0) circle (3pt);
    \draw[fill=black] (0,1) circle (3pt);
    \draw[fill=black] (0,-1) circle (3pt);
    \draw[fill=black] (2,0) circle (3pt);
    \draw[fill=black] (4,1) circle (3pt);
    \draw[fill=black] (4,-1) circle (3pt);
    \node at (-4,.5) {011011};
    \node at (-4,-1.5) {001001};
    \node at (-2,-.5) {010101};
    \node at (0,.5) {101101};
    \node at (0,-1.5) {010010};
    \node at (2,-.5) {010101};
    \node at (4,.5) {110110};
    \node at (4,-1.5) {100100};
    \draw[thick] (-4,1) -- (-2,0) -- (0,1) -- (2,0) -- (4,1);
    \draw[thick] (-4,-1) -- (-2,0) -- (0,-1) -- (2,0) -- (4,-1);
\end{tikzpicture}
    \caption{The unique size-8 component of $\cG_6^{(101)}$ from Table~\ref{differentsizes} (right).}
    \label{graphex}
\end{figure}
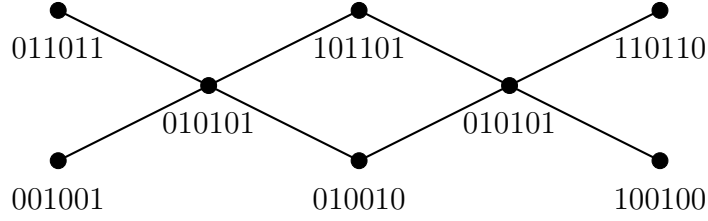

A first natural question is whether two keywords yield isomorphic graphs.
For instance, clearly $\cG_n^{(a)}$ and $\cG_n^{(\neg a)}$ are the same graph, since $a$ and $\neg a$ are logically interchangeable in Definition~\ref{def_graph}.
In addition to negation, the next definition names two other transformations we will consider.

\begin{definition} \label{def_seminegation}
    For $u=(u_1,\dots,u_n)\in\{0,1\}^n$, the \textit{reversal} of $u$ is
    \eeq{ \label{rev_def}
    \rev{u} \coloneqq (u_n,\dots,u_1).
    }
    The \textit{seminegation} of $u$ is the negation of only the letters in even positions:
    \eeq{ \label{semi_def}
    \semi{u} \coloneqq (u_1, \neg u_2, u_3, \neg u_4, \dots, (\neg)^{n-1}u_{n}).
    }
\end{definition}

By composing the operations of negation, seminegation, and reversal, one can produce from a single keyword a family of either $4$ or $8$ keywords that---according to the following result---all yield isomorphic graphs.

\begin{proposition} \label{thm_iso}
    For any keyword $a\in\{0,1\}^{m+1}$ and any $n\ge1$, \[\cG_n^{(a)}=\cG_n^{(\neg a)}\cong \cG_n^{(\rev a)} \cong \cG_n^{(\semi{a})}.\]
\end{proposition}

Our proof will invoke the following sufficient condition for two keywords to induce isomorphic graphs.

\begin{lemma}\label{generalf}
    Fix keywords $a,b\in\{0,1\}^{m+1}$.
    Suppose there exist bijections $f\colon\{0,1\}^n \to \{0,1\}^n$ and $\sigma\colon \llbrack 1,n\rrbrack \to \llbrack 1,n\rrbrack$ such that
    \eeq{ \label{fcondition}
    (f\circ\vphi_i^{(a)})(u) = (\vphi_{\sigma(i)}^{(b)}\circ f)(u)
    \quad \text{for all $u\in\{0,1\}^n$ and $i\in\llbrack1,n\rrbrack$.}
    }
    Then $d^{(a)}(u, v) = d^{(b)}(f(u),f(v))$ for all $u,v\in\{0,1\}^n$.
\end{lemma}


\begin{proof}
Induction on \eqref{fcondition} yields
\eeq{ \label{bcnex3}
f\circ\vphi_{i_r}^{(a)}\circ\cdots\circ\vphi_{i_1}^{(a)}
= \vphi_{\sigma(i_r)}^{(b)}\circ\cdots\circ\vphi_{\sigma(i_1)}^{(b)}\circ f.
}
We thus have the following implications:
\eq{ 
&\exists\ \text{$i_1,\dots,i_r$ such that $(\vphi_{i_r}^{(a)} \circ \dots \circ \vphi_{i_1}^{(a)})(u)=v$} \\
\implies\quad &\exists\ \text{$i_1,\dots,i_r$ such that $(f\circ \vphi_{i_r}^{(a)} \circ \dots \circ \vphi_{i_1}^{(a)})(u)=f(v)$} \\
\stackref{bcnex3}{\implies}\quad &
\exists\ \text{$i_1,\dots,i_r$ such that $(\vphi_{\sigma(i_r)}^{(b)} \circ \dots \circ \vphi_{\sigma(i_1)}^{(b)})(f(u))=f(v)$.}
}
Hence $d^{(a)}(u,v)\geq d^{(b)}(f(u),f(v))$. 
Since $f$ and $\sigma$ are invertible, we can interchange the roles of $a$ and $b$ to obtain the reverse inequality.
\end{proof}

\begin{proof}[Proof of Proposition~\ref{thm_iso}]
The equality $\cG_n^{(a)}=\cG_n^{(\neg a)}$ is immediate from the fact that $\vphi_i^{(a)} = \vphi_i^{(\neg a)}$ for any $a\in\{0,1\}^{m+1}$.

The isomorphism $\cG_n^{(a)}\cong \cG_n^{(\rev{a})}$ is obtained by taking $f(u)=\rev{u}$ and $\sigma(i) = n-i-m+1$ in Lemma~\ref{generalf}. 
To verify \eqref{fcondition}, we set $v = \vphi_i^{(a)}(u)$, $w = \vphi_{\sigma(i)}^{(\rev{a})}(\rev{u})$, and proceed to show $\rev{v} = w$. 
On one hand, the $j$-th letter of $\rev{v}$ is
\eq{
\rev{v}_j \stackref{rev_def}{=} v_{n-j+1}
&\stackref{def_simple_map}{=}\begin{cases}
    \neg u_{n-j+1} &\text{if $a\applies_i u$ and $n-j+1\in\llbrack i,i+m\rrbrack$} \\
    u_{n-j+1} &\text{otherwise}
\end{cases} \\
&\stackrefp{def_simple_map}{=}\begin{cases}
    \neg u_{n-j+1} &\text{if $a\applies_i u$ and $j\in\llbrack \sigma(i),\sigma(i)+m\rrbrack$} \\
    u_{n-j+1} &\text{otherwise.}
\end{cases}
}
On the other hand, the $j$-th letter of $w$ is
\eq{
w_j &\stackrefpp{def_simple_map}{rev_def}{=}\begin{cases}
\neg \rev{u}_j &\text{if $\rev{a}\applies_{\sigma(i)}\rev{u}$ and $j\in\llbrack\sigma(i),\sigma(i)+m\rrbrack$}\\
\rev{u}_j &\text{otherwise}
\end{cases} \\
&\stackref{rev_def}{=}\begin{cases}
\neg u_{n-j+1} &\text{if $\rev{a}\applies_{\sigma(i)}\rev{u}$ and $j\in\llbrack\sigma(i),\sigma(i)+m\rrbrack$}\\
u_{n-j+1} &\text{otherwise.}
\end{cases}
}
Comparing the two previous displays, we see that $\rev{v} = w$ provided the following equivalence is true:
\eeq{ \label{kbn823}
a\applies_i u \quad \iff \quad \rev{a}\applies_{\sigma(i)} \rev{u}.
}
To prove \eqref{kbn823}, we observe that the $(i+m)$-th letter of $u$ becomes the $\sigma(i)$-th letter of $\rev{u}$.
Hence
\eeq{
\label{ef8cx}
\rev{u_{\parng{i}{i+m}}} = \rev{u}_{\parng{\sigma(i)}{\sigma(i)+m}},
}
so we have
\begin{align*}
a\applies_i u \  &\iff \  u_{\parng{i}{i+m}}\in\{a,\neg a\} \\
&\iff \  \rev{u_{\parng{i}{i+m}}}\in\{\rev{a},\rev{\neg a}\} \qquad  \text{since reversal is a bijection} \\
&\iff \  \rev{u_{\parng{i}{i+m}}}\in\{\rev{a},\neg\rev a\} \qquad  \text{since reversal commutes with negation} \\
&\stackref{ef8cx}{\iff} \  \rev{u}_{\parng{\sigma(i)}{\sigma(i)+m}}\in\{\rev{a},\neg\rev a\}
\  \iff \ 
\rev{a}\applies_{\sigma(i)} \rev{u}.
\end{align*}

Finally, the isomorphism $\cG_n^{(a)}\cong \cG_n^{(\semi{a})}$ is obtained by taking $f(u) = \semi{u}$ and $\sigma(i) = i$ in Lemma~\ref{generalf}. 
To verify \eqref{fcondition}, we set $v = \vphi_i^{(a)}(u)$, $w = \vphi_i^{(\semi{a})}(\semi{u})$, and proceed to show $\semi{v} = w$.
On one hand, the $j$-th letter of $\semi{v}$ is
\eq{
(\semi{v})_j \stackref{semi_def}{=} (\neg)^{j-1}v_j
\stackref{def_simple_map}{=} \begin{cases}
(\neg)^ju_j &\text{if $a\applies_i u$ and $j\in\llbrack i,i+m\rrbrack$} \\
(\neg)^{j-1}u_j &\text{otherwise}.
\end{cases}
}
On the other hand, the $j$-th letter of $w$ is
\eq{
w_j &\stackrefpp{def_simple_map}{semi_def}{=} \begin{cases}
\neg(\semi{u})_j &\text{if $\semi{a}\applies_i \semi{u}$ and $j\in\llbrack i,i+m\rrbrack$} \\
(\semi{u})_j &\text{otherwise}.
\end{cases} \\
&\stackref{semi_def}{=}\begin{cases}
(\neg)^{j}u_j &\text{if $\semi{a}\applies_i \semi{u}$ and $j\in\llbrack i,i+m\rrbrack$} \\
(\neg)^{j-1}u_j &\text{otherwise}.
\end{cases}
}
Comparing the two previous displays, we see that $\semi{v} = w$ provided the following equivalence is true:
\eeq{ \label{uv711}
a\applies_i u \quad \iff \quad \semi{a}\applies_i \semi{u}.
}
To prove \eqref{uv711}, we compare a subword of the seminegation to the seminegation of the subword:
\eeq{ \label{ef7cx}
(\semi{u})_{\parng{i}{i+m}} = \begin{cases}
\semi{u_{\parng{i}{i+m}}} &\text{if $i$ is odd} \\
\neg(\semi{u_{\parng{i}{i+m}}}) &\text{if $i$ is even}.
\end{cases}
}
We now obtain \eqref{uv711} as follows:
\begin{align*}
a\applies_i u \  &\iff \  u_{\parng{i}{i+m}}\in\{a,\neg a\} \\
&\iff \  \semi{u_{\parng{i}{i+m}}}\in\{\semi{a},\semi{\neg a}\} \qquad  \text{since seminegation is a bijection on $\{0,1\}^{m+1}$} \\
&\iff \  \semi{u_{\parng{i}{i+m}}}\in\{\semi{a},\neg(\semi a)\} \qquad  \text{since seminegation commutes with negation} \\
&\stackref{ef7cx}{\iff} \  (\semi{u})_{\parng{i}{i+m}}\in\{\semi{a},\neg(\semi a)\}
\  \iff \ 
\semi{a}\applies_i \semi{u}. \qedhere
\end{align*}
\end{proof}

Our final result gives a restriction on the graphs obtainable from Definition~\ref{def_graph}.

\begin{proposition}\label{bipartite}
    For any keyword $a\in\{0,1\}^{m+1}$ and any $n\ge1$, $\cG_n^{(a)}$ is bipartite. 
\end{proposition}

\begin{proof}
We show $\cG_n^{(a)}$ contains no odd cycles.
Let $u\in \{0,1\}^n$ be a vertex in some cycle of length $r$.
That is, there exists a list of simple maps $\Phi = (\vphi_{i_r}^{(a)},\dots,\vphi_{i_1}^{(a)})$ that acts completely on $u$ (i.e.~satisfies \eqref{acts_completely}) and such that $\Phi(u)=(\vphi_{i_r}^{(a)} \circ\cdots\circ \vphi_{i_1}^{(a)})(u)=u$, with each simple map corresponding to one edge in the cycle. 
We will argue that $r$ must be even.

Take $i_{(1)}=\min\{i_1,\dots, i_r\}$. 
By minimality, the instances of $\vphi_{i_{(1)}}^{(a)}$ are the only simple maps in $\Phi$ that act on the $i_{(1)}$-th letter. 
Additionally, since $(\Phi(u))_{i_{(1)}}=u_{i_{(1)}}$, the application of $\Phi$ must negate the $i_{(1)}$-th letter an even number of times. 
Since $\Phi$ acts completely on $u$, it follows that there is an even number of instances of $\vphi_{i_{(1)}}^{(a)}$ in $\Phi$.

Let $i_{(q)}$ be the $q$-th smallest distinct index, i.e.
\eq{
i_{(q)}=\min\big\{\{i_1,\dots,i_r\}\setminus \{i_{(1)},\dots,i_{(q-1)}\}\big\},
}
whenever the minimum is taken over a nonempty set.
Assume inductively that for all $p<q$, there is an even number of instances of $\vphi_{i_{(p)}}^{(a)}$ in $\Phi$. 
Since $(\Phi(u))_{i_{(q)}}=u_{i_{(q)}}$, the application of $\Phi$ must negate the $i_{(q)}$-th letter an even number of times. 
The simple maps in $\Phi$ that act on the $i_{(q)}$-th letter all have index at most $i_{(q)}$. 
Because there are an even number of simple maps in $\Phi$ with index strictly less than $i_{(q)}$, there must also be an even number of instances of $\vphi_{i_{(q)}}^{(a)}$ (here we are again using the assumption that $\Phi$ acts completely on $u$).

We have thus shown that every distinct simple map in $\Phi = (\vphi_{i_r},\dots,\vphi_{i_1})$ must appear an even number of times. 
In particular, $r$ is even.
\end{proof}


\section{Open problems} \label{sec_open}

\begin{enumerate}[leftmargin=8mm,label=\textup{\ref*{sec_open}.\arabic*.},ref=\textup{\ref*{sec_open}.\arabic*}]

\item \label{open1} Does Proposition~\ref{thm_iso} exhaust all possible isomorphisms?
That is, if $\cG_n^{(a)} \cong \cG_n^{(b)}$ for all $n$, is it the necessarily the case that $b$ can be obtained from $a$ by some composition of negation, seminegation, and reversal?
The answer is \textit{yes} for $m\in\{1,2,3\}$, but there is evidence that the answer is \textit{no} for larger $m$.
Namely, when $a = 10000$ and $b = 01000$, we have verified by computer that $\cG_n^{(a)} \cong \cG_n^{(b)}$ for all $n\le 17$; see Table~\ref{evilcase} for the sizes of the equivalence classes when $n=17$.
Yet these two keywords are not linked by any combination of sequence of negation, seminegation, and reversal.

\begin{table}[ht]
\caption{Number of equivalence classes of size $s$ induced on $\{0,1\}^{17}$ by two different keywords ($a$ and $b$). 
The evidence suggests these two keywords yield identical counts for every $n$ (and even have isomorphic graphs), but this particular case falls outside the scope of Proposition~\ref{thm_iso}.}
    \centering
    \begin{tabular}{c||cc}
        $n=17$ & $a=10000$ & $b=01000$ \\\hline\hline
        $s=1$ & 46,498 & 46,498 \\
        $s=2$ & 28,308 & 28,308 \\
        $s=3$ & 3,344 & 3,344 \\
        $s=4$ & 3,730 & 3,730 \\
        $s=5$ & 154 & 154 \\
        $s=6$ & 312 & 312 \\
        $s=7$ & 4 & 4 \\
        $s=8$ & 42 & 42 \\\hline
        total & 82,392 & 82,392
    \end{tabular}
    \label{evilcase}
\end{table}

\item \label{open2} It is trivial that if $\cG_n^{(a)} \cong \cG_n^{(b)}$, then $\sim_a$ and $\sim_{b}$ induce the same number of size $s$ equivalence classes for every $s$.
Is the converse true?
Note that the example from Table~\ref{only1_table} shows that equality at $s=1$ is not sufficient for equality at other values of $s$.

\item \label{open3} Klarner \cite{klarner66} proved the solution set to $R(N) = s$ is infinite for all $s$ (and also identified these sets for $s = 1,2,3$).
This means that for any keyword $a\in\{110,001,100,011\}$, there are equivalence classes of every size (as $n\to\infty$).
For other keywords, which sizes are possible?

\item \label{open4} How does the maximal size of an equivalence class grow with $n$, and how does the growth rate vary with the keyword?
For the Fibonacci keyword $a = 1^{m}0$, Koc{\'a}bov{\'a}, Mas{\'a}kov{\'a}, and Pelantov{\'a} \cite{kocabova_masakova_pelantova05,kocabova_masakova_pelantova07} studied the related quantity
\eq{
\mathrm{Max}^{(m)}(n) \coloneqq \max\{R^{(m)}(N):\, F_n^{(m)} \le N < F_{n+1}^{(m)}\},
}
where $R^{(m)}(N)$ is the generalization of $R(N) = R^{(2)}(N)$ to $m$-step Fibonacci numbers.
For instance, one of their results \cite[Thm.~4.7]{kocabova_masakova_pelantova05} says
\eq{
\mathrm{Max}^{(2)}(2n+1) &= F_{n+1}^{(2)} \quad \text{for $n\ge0$}, \\
\mathrm{Max}^{(2)}(2n+2) &= \rlap{$2F_n^{(2)}$}\phantom{F_{n+1}} \quad \text{for $n\ge1$}.
}
Do other keywords yield faster or slower growth?

\end{enumerate}

\section{Acknowledgments}
Part of this research began at the 2017 Stanford Undergraduate Research Institute in Mathematics (SURIM), which included invaluable discussions with Patrick Revilla.  
We also thank George Schaeffer, the director of the 2017 SURIM, for feedback and encouragement.
E.B. and B.M. were partially supported by National Science Foundation grant DMS-2412473.

\bibliographystyle{myacm}
\bibliography{erikbib.bib}

\begin{thebibliography}{10}

\bibitem{baccherini_merlini_sprugnoli07}
{\sc Baccherini, D., Merlini, D., and Sprugnoli, R.}
\newblock Binary words excluding a pattern and proper {R}iordan arrays.
\newblock {\em Discrete Math. 307}, 9-10 (2007), 1021--1037.
\newblock \url{https://doi.org/10.1016/j.disc.2006.07.023}.

\bibitem{berstel01}
{\sc Berstel, J.}
\newblock An exercise on {F}ibonacci representations.
\newblock {\em Theor. Inform. Appl. 35}, 6 (2001), 491--498.
\newblock \url{https://doi.org/10.1051/ita:2001127}.

\bibitem{bicknelljohnson_fielder99}
{\sc Bicknell-Johnson, M., and Fielder, D.~C.}
\newblock The number of representations of {$N$} using distinct {F}ibonacci
  numbers, counted by recursive formulas.
\newblock {\em Fibonacci Quart. 37}, 1 (1999), 47--60.

\bibitem{bilotta_grazzini_pergola13}
{\sc Bilotta, S., Grazzini, E., and Pergola, E.}
\newblock Counting binary words avoiding alternating patterns.
\newblock {\em J. Integer Seq. 16}, 4 (2013), Art. 13.4.8, 17.

\bibitem{bilotta_merlini_pergola12}
{\sc Bilotta, S., Merlini, D., Pergola, E., and Pinzani, R.}
\newblock Pattern {$1^{j+1}0^j$} avoiding binary words.
\newblock {\em Fund. Inform. 117}, 1-4 (2012), 35--55.
\newblock \url{https://doi.org/10.3233/FI-2012-687}.

\bibitem{brown64}
{\sc Brown, Jr., J.~L.}
\newblock Zeckendorf's theorem and some applications.
\newblock {\em Fibonacci Quart. 2}, 3 (1964), 163--168.

\bibitem{carlitz68}
{\sc Carlitz, L.}
\newblock Fibonacci representations.
\newblock {\em Fibonacci Quart. 6}, 4 (1968), 193--220.

\bibitem{carrigan_hollars_rowland24}
{\sc Carrigan, J., Hollars, I., and Rowland, E.}
\newblock A natural bijection for contiguous pattern avoidance in words.
\newblock {\em Discrete Math. 347}, 3 (2024), Paper No. 113793, 8.
\newblock \url{https://doi.org/10.1016/j.disc.2023.113793}.

\bibitem{edson_zamboni04}
{\sc Edson, M., and Zamboni, L.~Q.}
\newblock On representations of positive integers in the {F}ibonacci base.
\newblock {\em Theoret. Comput. Sci. 326}, 1-3 (2004), 241--260.
\newblock \url{https://doi.org/10.1016/j.tcs.2004.06.025}.

\bibitem{edson_zamboni06}
{\sc Edson, M., and Zamboni, L.~Q.}
\newblock On the number of partitions of an integer in the {$m$}-bonacci base.
\newblock {\em Ann. Inst. Fourier (Grenoble) 56}, 7 (2006), 2271--2283.
\newblock \url{https://doi.org/10.5802/aif.2240}.

\bibitem{guibas_odlyzko81}
{\sc Guibas, L.~J., and Odlyzko, A.~M.}
\newblock String overlaps, pattern matching, and nontransitive games.
\newblock {\em J. Combin. Theory Ser. A 30}, 2 (1981), 183--208.
\newblock \url{https://doi.org/10.1016/0097-3165(81)90005-4}.

\bibitem{hoggatt_basin63}
{\sc Hoggatt, Jr., V.~E., and Basin, S.~L.}
\newblock Representations by complete sequences. part {I} ({F}ibonacci).
\newblock {\em Fibonacci Quart. 1}, 3 (1963), 1--14.

\bibitem{klarner66}
{\sc Klarner, D.~A.}
\newblock Representations of {$N$} as a sum of distinct elements from special
  sequences.
\newblock {\em Fibonacci Quart. 4}, 4 (1966), 289--306, 322.

\bibitem{klarner68}
{\sc Klarner, D.~A.}
\newblock Partitions of {$N$} into distinct {F}ibonacci numbers.
\newblock {\em Fibonacci Quart. 6}, 4 (1968), 235--244.

\bibitem{kocabova_masakova_pelantova05}
{\sc Koc{\'a}bov{\'a}, P., Mas{\'a}kov{\'a}, Z., and Pelantov{\'a}, E.}
\newblock Integers with a maximal number of {F}ibonacci representations.
\newblock {\em Theor. Inform. Appl. 39}, 2 (2005), 343--359.
\newblock \url{https://doi.org/10.1051/ita:2005022}.

\bibitem{kocabova_masakova_pelantova07}
{\sc Koc{\'a}bov{\'a}, P., Mas{\'a}kov{\'a}, Z., and Pelantov{\'a}, E.}
\newblock Ambiguity in the {$m$}-bonacci numeration system.
\newblock {\em Discrete Math. Theor. Comput. Sci. 9}, 2 (2007), 109--123.
\newblock \url{https://doi.org/10.46298/dmtcs.381}.

\bibitem{lekkerkerker52}
{\sc Lekkerkerker, C.~G.}
\newblock Voorstelling van natuurlijke getallen door een som van getallen van
  {F}ibonacci.
\newblock {\em Simon Stevin 29\/} (1952), 190--195.

\bibitem{merlini_sprugnoli11}
{\sc Merlini, D., and Sprugnoli, R.}
\newblock Algebraic aspects of some {R}iordan arrays related to binary words
  avoiding a pattern.
\newblock {\em Theoret. Comput. Sci. 412}, 27 (2011), 2988--3001.
\newblock \url{https://doi.org/10.1016/j.tcs.2010.07.019}.

\bibitem{noonan_zeilberger99}
{\sc Noonan, J., and Zeilberger, D.}
\newblock The {G}oulden-{J}ackson cluster method: extensions, applications and
  implementations.
\newblock {\em J. Differ. Equations Appl. 5}, 4-5 (1999), 355--377.
\newblock \url{https://doi.org/10.1080/10236199908808197}.

\bibitem{sedgewick_flajolet13}
{\sc Sedgewick, R., and Flajolet, P.}
\newblock {\em An {I}ntroduction to the {A}nalysis of {A}lgorithms}.
\newblock Addison-Wesley, 2013.

\bibitem{zeckendorf72}
{\sc Zeckendorf, E.}
\newblock Repr{\'e}sentation des nombres naturels par une somme de nombres de
  {F}ibonacci ou de nombres de {L}ucas.
\newblock {\em Bull. Soc. Roy. Sci. Li\`ege 41\/} (1972), 179--182.

\end{thebibliography}

\end{document}